\documentclass[12pt,reqno,a4paper]{amsart}
\usepackage{hyperref}
\usepackage{enumerate}
\usepackage[left=0.75in, right=0.75in, top= 0.9in, bottom=1.65in]{geometry}
\usepackage{amsmath,amssymb,amsthm,amsfonts}
\usepackage{graphicx}
\usepackage{tikz}

\usetikzlibrary{decorations.pathreplacing}
\usetikzlibrary{decorations.markings}
\usetikzlibrary{calc} 

\usetikzlibrary{graphs, graphs.standard}

\tikzset{
	modal/.style={>=stealth’,shorten >=1pt,shorten <=1pt,auto,node distance=1.5cm,
		semithick},
	world/.style={circle, draw,minimum size=.1cm,fill=gray!15},
	point/.style={circle,draw,inner sep=0.3mm,fill=black},
	circ/.style={circle,draw,inner sep=0.1mm,fill=white},
	reflexive above/.style={->,loop,looseness=7,in=120,out=60},
	reflexive below/.style={->,loop,looseness=7,in=240,out=300},
	reflexive left/.style={->,loop,looseness=7,in=150,out=210},
	reflexive right/.style={->,loop,looseness=7,in=30,out=330}
}

\usetikzlibrary{shapes}
\usetikzlibrary{plotmarks}
\usetikzlibrary{arrows}
\usetikzlibrary{positioning}
\theoremstyle{definition}
\newtheorem{defn}{Definition}[section]

\newtheorem{prop}[defn]{Proposition}
\newtheorem{thm}[defn]{Theorem}
\newtheorem{example}[defn]{Example}

\newtheorem{lem}[defn]{Lemma}

\newtheorem{claim}[defn]{Claim}
\newtheorem{subclaim}[defn]{Subclaim}
\title[List-distinguishing chromatic number of $K_{r,r+1}$-free graphs]
{An upper bound for the list-distinguishing chromatic number of $K_{r,r+1}$-free graphs}

\author{Amitayu Banerjee}
\address{E\"otv\"os Lor\'and University, Budapest, Hungary}
\email{banerjee.amitayu@gmail.com}

\date{}

\makeatletter
\@namedef{subjclassname@2020}{\textup{2020} Mathematics Subject Classification}
\makeatother
\subjclass[2020]{05C15, 05C25.}

\keywords{Distinguishing chromatic number, List-distinguishing chromatic number, Cayley graphs}

\begin{document}

\begin{abstract}
The {\em distinguishing chromatic number} of a graph $G$, denoted by $\chi_{D}(G)$, is the least integer $k$ such that there exists a proper vertex coloring of $G$ with $k$ colors that is preserved only by the identity automorphism. Collins and Trenk proved that $\chi_{D}(G)\leq 2 \Delta(G)$ for any connected graph $G$ and that the equality holds if $G$ is the complete bigraph $K_{\Delta(G),\Delta(G)}$ or $C_{6}$ where $\Delta(G)$ is the maximum degree of $G$.
Let $4\leq r\leq \frac{\Delta(G)}{2}$.
Inspired by Catlin's combinatorial techniques, we prove that if $G$ is an $(n-2r+1)$-connected graph of order $n>2r$, and 
$G$ is $K_{r,r+1}$-free, then $\chi_{D}(G)\leq \chi_{D_{L}}(G)\leq 2\Delta(G)- (3\lfloor\frac{(\Delta(G)+2)}{r+1}\rfloor-4)$ where $\chi_{D_{L}}(G)$ is the list-distinguishing chromatic number of $G$.
We obtain applications of our result to families of Cayley-type graphs.
\end{abstract}

\maketitle
\section{Introduction} 
All graphs considered here are finite, simple, and undirected. Let $K_{n}$ be a complete graph of $n$ vertices, $K_{t_{1},...,t_{m}}$ be a complete $m$-partite graph with partitions of size $t_{1},...,t_{m}$, and $C_{n}$ be a cycle of $n$ vertices.
Collins and Trenk \cite{CT2006} introduced the distinguishing chromatic number of a graph and obtained a general upper bound for it.

\begin{thm}\label{Theorem 1.1}{(Collins and Trenk \cite[Theorem 4.5]{CT2006})}
{\em Let $G$ be a connected graph. Then $\chi_{D}(G)\leq 2\Delta(G)-1$ if $G\not\in\{K_{\Delta(G),\Delta(G)},C_{6}\}$ and $\chi_{D}(G) = 2\Delta(G)$} if $G\in\{K_{\Delta(G),\Delta(G)},C_{6}\}$.    
\end{thm}

\begin{defn}\label{Definition 1.2}
Given a family of graphs $H_{1},..., H_{k}$, a graph $G$ is {\em $(H_{1},...,H_{k})$-free} if none of the graphs $H_{1}$,...,$H_{k}$ is an induced subgraph of $G$.    
\end{defn}

A central theme in the study of the distinguishing chromatic number is to determine either 
the exact value of $\chi_{D}(G)$ or to estimate the upper bound of $\chi_{D}(G)$ in terms of the maximum degree 
$\Delta(G)$ for specific graph classes. Accordingly, $\chi_{D}(G)$ has been 
investigated for several classes of graphs, including connected bipartite graphs 
\cite{LS2009}, 3-connected planar graphs \cite{FNS2011}, graphs of girth at least $5$ 
\cite{AS2016,Cra2018}, Cayley graphs \cite{BPS2019}, connected $(C_3,C_4)$-free graphs 
\cite{Cra2018}, and graph operators such as the middle graph \cite{BGMa2025}.
Recently, Brause et al. \cite{BKPS2025} showed that the upper bound of $\chi_{D}(G)$ mentioned in Theorem \ref{Theorem 1.1} can be substantially reduced if one forbids certain small graphs as induced subgraphs of $G$. 
In 2013, Ferrara et al. \cite{FGHSW2013} extended the notion of distinguishing proper coloring to a list distinguishing proper coloring. 

\begin{defn}\label{Definition 1.3}
Given an assignment $L = {L(v)}_{v\in V(G)}$ of lists of available colors to the vertices of $G$, we say that $G$ is {\em properly L-distinguishable} if there is a proper coloring $f$ of $G$ such that
$f$ is preserved only by the identity automorphism and
$f(v) \in L(v)$ for all $v\in V(G)$. 
The {\em list-distinguishing chromatic number} of $G$, denoted by $\chi_{D_{L}}(G)$, is the minimum integer $k$ such that $G$ is properly $L$-distinguishable for any list
assignment $L$ with $\vert L(v)\vert = k$ for all $v\in V(G)$.
\end{defn}


\subsection{Previous results on list-distinguishing chromatic number}
It is clear that $\chi(G)\leq \chi_{D}(G) \leq \chi_{D_{L}}(G)$ where $\chi(G)$ is the chromatic number of $G$. In Banerjee et al. \cite[Section 1.1]{BMG2025}, we constructed a connected graph $G'$ with $\chi_{D_L}(G')\neq \chi_D(G')$ and established the following:

\begin{thm}\label{Theorem 1.4}
{(Banerjee, Moln\'{a}r, and Gopaulsingh \cite[Theorem 3.1]{BMG2025})} 
{\em Let $G$ be a connected finite graph with maximum degree $\Delta(G)$. Then $\chi_{D_{L}}(G)\leq 2\Delta(G)-1$, unless $G$ 
is $K_{\Delta(G),\Delta(G)}$ or $C_6$. In these cases, $\chi_{D_{L}}(G) = 2\Delta(G)$. } 
\end{thm}
 
Borodin and Kostochka \cite{BK1977}, Catlin \cite{Cat1978}, and Lawrence \cite{Law1978} independently improved the Brooks' upper bound for $\chi (G)$ for graphs omitting certain cliques.  
They showed that if $K_{r+1} \not\subseteq G$, where $3 \leq r \leq \Delta(G)$, then $\chi(G) \leq \frac{r}{r+1}(\Delta(G) + 2)$ by applying the following result of Lov\'{a}sz.

\begin{lem}\label{Lemma 1.5}{(Lov\'{a}sz; \cite{Lov1966})}
{\em If $\sum_{i=1}^{t} x_{i}\geq \Delta(G)+1-t$, then there is a partition of $V(G)$ into sets $V_{1},...,V_{t}$ such that $\Delta(G[V_{i}])\leq x_{i}$ for $1\leq i\leq t$, if $G[V_{i}]$ is the subgraph of $G$ induced by $V_{i}$.}    
\end{lem}

In \cite{BMG2025}, we applied Lemma \ref{Lemma 1.5} and followed the methods of Catlin \cite{Cat1978} to reduce the upper bound of $\chi_{D_{L}}(G)$ from Theorem \ref{Theorem 1.4} for certain graphs and proved the following result.

\begin{thm}\label{Theorem 1.6}
{(Banerjee, Moln\'{a}r, and Gopaulsingh \cite[Theorem 4.3]{BMG2025})} 
{\em Fix $6\leq t\leq \Delta(G)$. Let $G$ be a $(n-t)$-connected graph of order $n$ such that $G$ is $K_{t,t}$-free.
Then, 
\begin{center}
$\chi_{D_{L}}(G)\leq 2\Delta(G)-  (3\lfloor\frac{(\Delta(G)+1)}{t+1}\rfloor-2)$.    
\end{center}
}
\end{thm}

\subsection{Main Result}
In \cite{Cat1978a}, Catlin used an intriguing combinatorial trick instead of applying Lemma \ref{Lemma 1.5} to prove that if $G$ contains no $K_{r+2}\backslash \{e\}$ as a subgraph,\footnote{We follow the notations of Catlin \cite{Cat1978a}, where $K_{r+2}\backslash \{e\}$ denotes the clique on $r+2$ vertices, minus an edge.} where $3\leq r$, then $\chi(G)\leq \frac{r}{r+1}(\Delta(G)+3)$.
In this manuscript, inspired by the methods of Catlin \cite{Cat1978a}, 
we obtain an improved upper bound for $\chi_{D_{L}}(G)$ from Theorem \ref{Theorem 1.4} for a class of connected graphs.


\begin{thm}\label{Theorem 1.7}{\em Fix $4\leq r\leq \frac{\Delta(G)}{2}$. Let $G$ be a $(n-2r+1)$-connected graph of order $n>2r$ such that
$G$ is $K_{r,r+1}$-free.
Then, 
\begin{center}
$\chi_{D_{L}}(G)\leq 2\Delta(G)-  (3\lfloor\frac{(\Delta(G)+2)}{r+1}\rfloor-4)$.
\end{center}
}
\end{thm}

\subsection{Applications}
In Section~4, we verify the hypotheses of Theorem~\ref{Theorem 1.7} for explicit graph families constructed from Paley graphs and a family of Cayley graphs over dihedral groups and show that $\chi_{D_{L}}(G)\leq 2\Delta(G)-5$ for such graphs $G$.

\section{Definitions and useful facts}

\begin{defn}\label{Definition 2.1}
Let $G$ be a graph with the vertex set $V(G)$ and the edge set $E(G)$.
\begin{enumerate}
\item $G$ is {\em $k$-connected} if it has at least $k+1$ vertices and removal of any $k-1$ or fewer vertices leaves a connected graph. 
The connectivity of $G$ is the largest $k$ for which $G$ is $k$-connected.
Let $\kappa(G)$ be the connectivity of the graph $G$.

\item The {\em degree} of a vertex $v \in V(G)$ in $G$, denoted by $\deg_{G}(v)$, is the number of edges that emerge from $v$.
We say $G$ is a {\em $k$-regular graph} if all vertices of $G$ have degree $k$. The degree of a regular graph $G$, denoted by $d(G)$, is the same as the maximum degree $\Delta(G)$ and the minimum degree $\delta(G)$.

\item $G$ is an {\em edge-transitive graph} if given any two edges $e_{1}$ and $e_{2}$ of $G$, there is an automorphism of $G$ that maps $e_{1}$ to $e_{2}$. 

\item An {\em independent set} of $G$ is a set of vertices of $G$, no two of which are adjacent vertices. 
The independence number of $G$, denoted by $\alpha(G)$, is the size of the largest independent set of $G$.

\item The {\em join} of two graphs $G$ and $H$, denoted by $G\vee H$, is obtained by taking $G$ and $H$ and adding edges between every vertex of $G$ and every vertex of $H$.
\end{enumerate}
\end{defn}


\begin{lem}{(Hoffman; \cite[Theorem 1]{Hae2021})}\label{Lemma 2.2}
{\em The independence number $\alpha(G)$ for a $k$-regular graph $G$ on $n$ vertices with the least eigenvalue $\lambda_{min}$ is at most
$n\cdot\frac{-\lambda_{\min}}{k-\lambda_{\min}}$.}   
\end{lem}

\begin{defn}{(Undirected Cayley graph)}\label{Definition 2.3}
Let $\Gamma$ be a finite group and let $S \subseteq \Gamma \setminus \{e\}$ 
be an inverse-closed subset of $\Gamma \setminus \{e\}$ i.e., $S = S^{-1}$ , where 
$S^{-1}:= \{s^{-1}: s \in S\}$.
The \emph{undirected Cayley graph} Cay$(\Gamma, S)$ is the graph with a set of vertices $\Gamma$, and the vertices $u$ and $v$ are adjacent in Cay$(\Gamma, S)$ if and only if $uv^{-1} \in S$. 
\end{defn}

It is known that Cay$(\Gamma, S)$ is a $|S|$-regular graph, and Cay$(\Gamma, S)$ is
connected if and only if $S$ is a generating set of $\Gamma$.

\begin{defn}{(Paley graph)}\label{Definition 2.6}
Let $q$ be a prime power with $q \equiv 1 \pmod{4}$. The \emph{Paley graph} $P(q)$ is the undirected Cayley graph of the additive group $(\mathbb{F}_{q},+)$ with generating set equal to the set of non-zero quadratic residues in $\mathbb{F}_{q}$. The vertex set of $P(q)$ is $\mathbb{F}_{q}$ and two distinct vertices $x,y \in \mathbb{F}_{q}$ are adjacent if and only if $x-y$ is a non-zero quadratic residue in $\mathbb{F}_{q}^{\times}$.
\end{defn}

\begin{lem}{(Watkins; \cite[Corollary 1A]{Wat1970})}\label{Lemma 2.7}
{\em If $G$ is a connected edge-transitive graph, then $\kappa(G)=\delta(G)$}. 
\end{lem}

The Paley graph $P(q)$ is a connected, edge transitive graph. By Lemma \ref{Lemma 2.7}, $\kappa(P(q))=\delta(P(q))$.

\section{The proof of Theorem 1.7}
Suppose $G$ is a $(n-2r+1)$-connected graph of order $n$ such that
$G$ does not contain $K_{r,r+1}$ as an induced subgraph for a fixed $4\leq r\leq \frac{\Delta(G)}{2}$.

\subsection{Defining $t$ and $h_{i}$ for each $1\leq i\leq t$} We define 
\begin{itemize}
    \item $h=\Delta(G)$,
    \item $t=\lfloor \frac{h+2}{r+1}\rfloor$, 
    \item $h_{i}=r$ for all $1\leq i\leq t-1$, and 
    \item $h_{t}$ is an integer that satisfies $r\leq h_{t}=h+2-t(r+1)+r\leq 2r$.
\end{itemize}
Thus, $h_{1}+...+h_{t-1}+h_{t}=(t-1)r+h+2-t(r+1)+r=h-t+2$.

\subsection{Defining a partition $(X_{1},...,X_{t})$ of $V(G)$} For a subset $X\subseteq V(G)$ of vertices of a graph $G$, let $G[X]$ denote the subgraph of $G$ induced by $X$. If $(X_{1},...,X_{t})$ is a partition of $V(G)$, then we define the integer-valued function $f$ as follows:
\begin{center}
$f(X_{1},...,X_{t})=h_{1}\vert X_{1}\vert + h_{2}\vert X_{2}\vert +... + h_{t}\vert X_{t}\vert-(\vert E(G[X_{1}])\vert +...+ \vert E(G[X_{t}])\vert)$.
\end{center}

Consider the partition $(X_{1},...,X_{t})$ of $V(G)$ that
\begin{enumerate}
    \item maximizes $f(X_{1},...,X_{t})$;
    \item minimizes the total number of complete bipartite induced subgraphs $K_{r,r}$ in the $G[X_{i}]$'s in case $h_{i}=r$ and $(1)$ holds.
\end{enumerate}

The following lemma is implicit in the proof of Catlin \cite{Cat1978a}. For the reader's convenience, the complete proof is included.

\begin{lem}\label{Lemma 3.1}
\emph{For every $1\le i\le t$, we have
$
\Delta(G[X_i])\le h_i.
$}
\end{lem}

\begin{proof}
Let $x\in X_1$. Since $(X_1,\ldots,X_t)$ is a partition of $V(G)$, for each $2\leq i\leq t$, the tuple
$\bigl(X_1\setminus\{x\},X_2,\ldots,X_i\cup\{x\},\ldots,X_t\bigr)$
is also a partition of $V(G)$. By the maximality of $f(X_1,\ldots,X_t)$, we have
\[
0\leq f(X_1,\ldots,X_t)
-f\bigl(X_1\setminus\{x\},X_2,\ldots,X_i\cup\{x\},\ldots,X_t\bigr).
\]
This expands as
\[
\begin{aligned}
0\leq {}&
h_1|X_1|-h_1(|X_1|-1)
+h_i|X_i|-h_i(|X_i|+1)\\
&-|E(G[X_1])|+|E(G[X_1\setminus\{x\}])|-|E(G[X_i])|+|E(G[X_i\cup\{x\}])|.
\end{aligned}
\]
Simplifying, we obtain 
$0\leq h_1 - h_i - \deg_{G[X_1]}(x) + \deg_{G[X_i \cup \{x\}]}(x).$
Hence, it follows that
\[
\deg_{G[X_1]}(x) \leq h_1 - h_i + \deg_{G[X_i \cup \{x\}]}(x) \text{ for each } 2\leq i\leq t.     
\]
Moreover, $\deg_{G[X_{1}]}(x)= \deg_{G[X_1 \cup \{x\}]}(x)=h_{1}-h_{1}+\deg_{G[X_1 \cup \{x\}]}(x)$.
Combining this with the above inequalities for $i=2,\ldots,t$ and summing over
$i=1,\ldots,t$, we obtain
\[
t \cdot \deg_{G[X_1]}(x) \leq t h_1 - \sum_{i=1}^{t} h_i + \sum_{i=1}^{t} \deg_{G[X_i \cup \{x\}]}(x).
\]
We note that $h=\Delta(G)$,
$\sum_{i=1}^{t} \deg_{G[X_i \cup \{x\}]}(x) = \deg_G(x)$,
and \( \sum_{i=1}^{t} h_i = h - t + 2 \). Thus, 
\[
t \cdot \deg_{G[X_1]}(x) \leq t h_1 - (h - t + 2) + \deg_G(x) \leq t h_1 + t - 2.
\]

Consequently, $\deg_{G[X_1]}(x) \leq h_1 + \frac{t - 2}{t}$.
Since both \( \deg_{G[X_1]}(x) \) and \( h_1 \) are integers, it follows that
$
\deg_{G[X_1]}(x) \leq h_1.
$
The same argument, with $X_i$ in place of $X_1$, shows that $\deg_{G[X_i]}(x) \le h_i$ for every vertex $x \in X_i$ where $1\leq i \leq t$.
\end{proof}


\subsection{Main proof}

\begin{claim}\label{claim 3.2}
{\em If $h_{t}>r$, then $\chi_{D_{L}}(G[X_{t}])\leq 2h_{t}-1$.} 
\end{claim}

\begin{proof}

\noindent \textsc{Case 1:}  Suppose $\vert V_{G[X_{t}]}\vert< 2r$.

As coloring the vertices with distinct colors from their respective lists of size $\vert V_{G[X_{t}]}\vert$ yields a proper distinguishing coloring we have
\[\chi_{D_{L}}(G[X_{t}])\leq \vert V_{G[X_{t}]}\vert \leq 2r-1\leq 2(h_{t}-1)-1< 2h_{t}-1.\] 

\noindent \textsc{Case 2:} Suppose $\vert V_{G[X_{t}]}\vert\geq 2r$. Since $G$ is $(n-2r+1)$-connected, we have that $G[X_{t}]$ is connected.
We note that $\Delta(G[X_{t}])\leq h_{t}$.

\vspace{2mm}
\begin{enumerate}
    \item[] \textsc{Subcase 2.1:} Let $\Delta(G[X_{t}])< h_{t}$.
Then, by Theorem \ref{Theorem 1.4}, we have
\[\chi_{D_{L}}(G[X_{t}])\leq 2(\Delta (G[X_{t}]))\leq 2(h_{t}-1)=2h_{t}-2.\]
    \item[] \textsc{Subcase 2.2:}
    Let $\Delta(G[X_{t}])= h_{t}>r$.
Since $G$ is $K_{r,r+1}$-free, $G$ contains no $K_{r+1,r+1}$ as an induced subgraph (and hence no $K_{h_{t}, h_{t}}$), neither does $G[X_{t}]$. 
Thus, $G[X_{t}]\neq K_{\Delta(G[X_{t}]), \Delta(G[X_{t}])}$. 
Since we assumed $r\geq 4$, we have $G[X_{t}]\neq C_{6}$ as $\vert V_{G[X_{t}]}\vert\geq 2r\geq 8$.
Thus, by Theorem \ref{Theorem 1.4}, we have
\[\chi_{D_{L}}(G[X_{t}])\leq 2\Delta(G[X_{t}])-1=2h_{t}-1.\]
\end{enumerate} 
This finishes the proof of Claim \ref{claim 3.2}.
\end{proof}

\begin{claim}\label{claim 3.3}
{\em $\chi_{D_{L}}(G[X_{i}])\leq 2h_{i}-1$ for all $1\leq i\leq t$.}
\end{claim}

\begin{proof}
In view of Claim \ref{claim 3.2}, we may assume that $h_{i}=r$ for all $1\leq i\leq t$.
We fix $1\leq i\leq t$.
\noindent \textsc{Case 1:}  If $\vert V_{G[X_{i}]}\vert< 2r$, then
$\chi_{D_{L}}(G[X_{i}])\leq \vert V_{G[X_{i}]}\vert \leq 2r-1= 2h_{i}-1.$ 

\noindent \textsc{Case 2:} If $\vert V_{G[X_{i}]}\vert\geq 2r$, then $G[X_{i}]$ is connected.
We recall that $\Delta(G[X_{i}])\leq h_{i}$.

\vspace{2mm}
\begin{enumerate}
    \item[] \textsc{Subcase 2.1:} Let $\Delta(G[X_{i}])< h_{i}$.
Then, by Theorem \ref{Theorem 1.4}, we have
\[\chi_{D_{L}}(G[X_{i}])\leq 2(\Delta (G[X_{i}]))\leq 2(h_{i}-1)=2h_{i}-2.\]
    \item[] \textsc{Subcase 2.2:}
    Let $\Delta(G[X_{i}])= h_{i}=r$.

\begin{subclaim}\label{Subclaim 3.4}
{\em There is no induced subgraph $K_{\Delta(G[X_{i}]), \Delta(G[X_{i}])}$ in $G[X_{i}]$.}    
\end{subclaim}

\begin{proof}
For the sake of contradiction, suppose that $G[X_{i}]$ contains an induced subgraph $C_{0}=K_{\Delta(G[X_{i}]), \Delta(G[X_{i}])}=K_{h_{i},h_{i}}=K_{r,r}$. Let $x_{0}\in C_{0}$. 
\vspace{2mm}

\textbf{Step 1:} 
We note that $x_{0}$ is adjacent to  $\Delta(G[X_{i}])=h_{i}$ vertices in $C_{0}=K_{h_{i}, h_{i}}$.
{\em We claim there exists $j, 1\leq j\leq t$ such that $j\neq i$, and $x_{0}$ is adjacent to $h_{j}$ or fewer vertices in $X_{j}$.} 
Otherwise, if $x_{0}$ is adjacent to at least $h_{j}+1$ vertices in each $X_{j}$ for $j\neq i$, then 

\begin{center}
$\deg_{G}(x_{0})\geq h_{i}+\sum_{j\neq i, 1\leq j\leq t}(h_{j}+1)=\sum_{1\leq p\leq t}h_{p}+(t-1)=(h-t+2)+t-1=h+1>h$,     
\end{center}

since $\sum_{1\leq p\leq t}h_{p}=h-t+2$, which is a contradiction to the assumption that $h=\Delta(G)$. We fix such a $j$. 
\vspace{2mm}

\textbf{Step 2:} Suppose, we move $x_{0}$ from $X_{i}$ to $X_{j}$. Then the maximality of $f$ is preserved as
$f(X_{1},..., X_{t}) = f(X_{1},.., X_{i}-x_{0},..., X_{j}+x_{0},.., X_{t})$.
\vspace{2mm}

\newcommand{\rsize}{5} 

\begin{figure}[h]
\centering

\begin{tikzpicture}[scale=0.78, every node/.style={inner sep=0.8pt}]

\def\W{12}   
\def\H{4.5}  
\def\sqW{\W/4} 

\draw[thick] (0,0) rectangle (\W,\H);
\foreach \i in {1,2,3} {
  \draw[thick] (\i*\sqW,0) -- (\i*\sqW,\H);
}

\def\inset{0.3}
\coordinate (A0) at (\inset,\inset);
\coordinate (B0) at (\sqW-\inset,\H-\inset);

\draw[thick,rounded corners=12pt] (A0) rectangle (B0);

\node[font=\small, anchor=south, yshift=-6mm, xshift=+7mm] 
  at ($ (A0)!0.5!(B0) $ |- B0) {$C_{0}$};

\pgfmathsetmacro{\leftx}{\inset+0.25*(\sqW-2*\inset)}   
\pgfmathsetmacro{\rightx}{\inset+0.75*(\sqW-2*\inset)}
\pgfmathsetmacro{\topY}{\H-\inset-0.2}   
\pgfmathsetmacro{\bottomY}{\inset+0.2}

\pgfmathtruncatemacro{\denom}{\rsize-1}
\ifnum\denom<1
  \pgfmathsetmacro{\denom}{1}
\fi

\foreach \i in {1,...,\rsize} {
  \pgfmathsetmacro{\y}{\bottomY + (\topY-\bottomY)*(\rsize-\i)/\denom}
  \node[draw,circle,fill=white,inner sep=1pt] (L\i) at (\leftx,\y) {};
  \node[draw,circle,fill=white,inner sep=1pt] (R\i) at (\rightx,\y) {};
}
\node[below left=1.5pt and -2pt] at (L1.south west) {$x_{0}$};

\foreach \i in {1,...,\rsize} {
  \foreach \j in {1,...,\rsize} {
    \draw (L\i) -- (R\j);
  }
}

\coordinate (A1) at (2*\sqW+\inset,\inset);
\coordinate (B1) at (3*\sqW-\inset,\H-\inset);

\draw[thick,rounded corners=12pt] (A1) rectangle (B1);

\node[font=\small, anchor=south, yshift=-6mm, xshift=7mm] 
  at ($ (A1)!0.5!(B1) $ |- B1) {$C_{1}$};

\pgfmathsetmacro{\leftxC}{2*\sqW+\inset+0.25*(\sqW-2*\inset)}   
\pgfmathsetmacro{\rightxC}{2*\sqW+\inset+0.75*(\sqW-2*\inset)}
\pgfmathsetmacro{\topYC}{\H-\inset-0.2}   
\pgfmathsetmacro{\bottomYC}{\inset+0.2}

\foreach \i in {1,...,\rsize} {
  \pgfmathsetmacro{\y}{\bottomYC + (\topYC-\bottomYC)*(\rsize-\i)/\denom}
  \node[draw,circle,fill=white,inner sep=1pt] (Lc\i) at (\leftxC,\y) {};
  \node[draw,circle,fill=white,inner sep=1pt] (Rc\i) at (\rightxC,\y) {};
}
\node[below left=2pt and -2pt] at (Lc1.south west) {$x_{0}$};

\foreach \i in {1,...,\rsize} {
  \foreach \j in {1,...,\rsize} {
    \draw (Lc\i) -- (Rc\j);
  }
}
\node[font=\small, anchor=south, yshift=-25mm, xshift=0mm] 
  at ($ (A1)!0.5!(B1) $ |- B1) {$G[X_{j}\cup \{x_{0}\}]$};
  
\node[font=\small, anchor=south, yshift=-25mm, xshift=0mm] 
  at ($ (A0)!0.5!(B0) $ |- B0) {$G[X_{i}\backslash \{x_{0}\}]$};
\draw[->,thick,bend left=12] (L1) to (Lc1);

\end{tikzpicture}
\caption{\em After vertex $x_{0}$ is moved from $X_{i}$ to $X_{j}$, $C_{0}\backslash \{x_{0}\}=K_{r-1,r}$ is left in $G[X_{i}\backslash\{x_{0}\}]$, and $C_{1}=K_{r,r}$ is formed in $G[X_{j}\cup \{x_{0}\}]$ when $r=5$.}
\label{Figure 1}
\end{figure}

\textbf{Step 3:} We obtain the desired contradiction to the assumption that $G[X_{i}]$ contains $C_{0}$. To avoid violating condition (2), i.e. the partition $(X_{1},..., X_{t})$ was chosen to minimize the total number of induced subgraph $K_{r,r}$ in the $G[X_{i}]$'s in case $h_{i}=r$ and $f(X_{1},..., X_{t})$ is maximum, 
the destruction of $C_{0}=K_{r,r}$ in $G[X_{i}\backslash \{x_{0}\}]$ must be accompanied by the formation of another induced subgraph $K_{r,r}$ in $G[X_{j}\cup \{x_{0}\}]$, say $C_{1}$. 
Then the complete bipartite graph $C_{0}\backslash\{x_{0}\}=K_{r-1,r}$ is left behind in $G[X_{i}\backslash \{x_{0}\}]$ (See Fig. \ref{Figure 1}).

We repeat this process by picking a vertex $x_{1}\neq x_{0}$ in $C_{1}$ and removing it from $G[X_{j}\cup \{x_{0}\}]$ to create another induced subgraph $K_{r,r}$ in $G[X_{k}\cup \{x_{1}\}]$ for some $1\leq k\leq t$, $k\neq j$, say $C_{2}$. This leaves behind a complete bipartite graph $K_{r-1,r}$ in $G[(X_{j}\cup \{x_0\})\backslash \{x_{1}\}]$. 
Continuing in this fashion, we will obtain a sequence $\{C_{i}\}_{i\geq 0}$.
Since $G$ is finite, the sequence $\{C_{i}\}_{i\geq 0}$ cannot consist of pairwise distinct members indefinitely. 
Hence, there exist integers $0\leq k<m$ such that $C_m=C_k$.
We choose the first such repetition, so that $C_0,\ldots,C_{m-1}$ are pairwise distinct.
Now $x_k\in C_k=C_m$. 
Since $x_k$ was the vertex removed from $C_k$ when $C_{k+1}$ was formed, the graph
$C_k\setminus\{x_k\}$
is an induced $K_{r-1,r}$. 
Moreover, $x_{m-1}\in C_m=C_k$. 
Without loss of generality, choose $x_k$ from $C_k=C_m$ such that $x_{m-1}$ and $x_k$ belong to the same bipartition class of $C_k$, where $x_{m-1}\neq x_k$. 



\begin{figure}[h]
\centering

\begin{tikzpicture}[scale=0.8]

\def\sqsize{4} 
\def\yshift{0.3} 
\draw[thick] (0.7,0.1) rectangle (3.3,3.3);

\node[anchor=north west] at (3.5, 3.5) {$C_{k}\setminus \{x_{k}\}=K_{r-1,r}$};

\foreach \i in {1,...,4} {
  \fill (1,\sqsize+\yshift-0.5-\i*0.7) circle (2pt);
}

\foreach \j in {1,...,5} {
  \fill (\sqsize-1,\sqsize+\yshift-0.5-\j*0.7) circle (2pt);
}

\foreach \i in {1,...,4} {
  \foreach \j in {1,...,5} {
    \draw (1,\sqsize+\yshift-0.5-\i*0.7) -- (\sqsize-1,\sqsize+\yshift-0.5-\j*0.7);
  }
}

\fill (1,\sqsize+\yshift-0.5) circle (2pt) node[left=1pt] {$x_{k}$};
\fill (1,\sqsize+\yshift+0.2) circle (2pt) node[left=1pt] {$x_{m-1}$};

\foreach \j in {1,...,5} {
  \draw (1,\sqsize+\yshift-0.5) -- (\sqsize-1,\sqsize+\yshift-0.5-\j*0.7);
  \draw (1,\sqsize+\yshift+0.2) -- (\sqsize-1,\sqsize+\yshift-0.5-\j*0.7);
}
\end{tikzpicture}
\caption{\em Formation of the induced subgraph $K_{r+1,r}$ of $G$ when $r=5$.}
\label{Figure 2}
\end{figure}

Then both $x_k$ and $x_{m-1}$ are adjacent to all $r$ vertices in the other bipartition class of $C_k$, while $x_k$ and $x_{m-1}$ are nonadjacent. Consequently,
\[
\bigl(C_k\setminus\{x_k\}\bigr)\cup\{x_k,x_{m-1}\}
\]
induces a $K_{r+1,r}$ in $G$. This contradicts the hypothesis that $G$ contains no induced subgraph isomorphic to $K_{r,r+1}$ (see Fig. \ref{Figure 2}).
Thus, there is no induced subgraph $K_{\Delta(G[X_{i}]), \Delta(G[X_{i}])}$ in $G[X_{i}]$.
This proves the subclaim.
\end{proof}

By Subclaim \ref{Subclaim 3.4}, we have $G[X_{i}]\neq K_{\Delta(G[X_{i}]), \Delta(G[X_{i}])}$. 
Moreover, $G[X_{i}]\neq C_{6}$ as $\vert V_{G[X_{i}]}\vert\geq 2r\geq 8$.
By Theorem \ref{Theorem 1.4}, we have
\[\chi_{D_{L}}(G[X_{i}])\leq 2\Delta(G[X_{i}])-1=2h_{i}-1.\]
\end{enumerate}
This finishes the proof of Claim \ref{claim 3.3}.
\end{proof}

\begin{defn}\label{Definition 3.5}
We denote init$_{k}(L)$ by the first $k$ elements of a list $L$.    
\end{defn}


\begin{claim}\label{claim 3.6}
$\chi_{D_{L}}(G)\leq 2\Delta(G)-(3t-4).$
\end{claim}

\begin{proof}
First, we show that $\chi_{D_{L}}(G)\leq \sum_{i=1}^{t} \chi_{D_{L}}(G[X_{i}])$. Let $k_{i}=\chi_{D_{L}}(G[X_{i}])$ for each $1\leq i\leq t$. Assign any list $L(v)$ to each vertex $v \in V(G)$ such that $\vert L(v)\vert= \sum_{i=1}^{t}k_{i}$.
We construct a coloring $f$ for the vertices of $G$ using the first $k_{1}$ colors from each list to color $G[X_{1}]$, then the next $k_{2}$ colors from each list to color $G[X_{2}]$, and so on.  
Formally, $f$ is defined as follows:

\begin{enumerate}
    \item Let $L^{1}(v):=$ init$_{k_{1}}(L(v))$ for each vertex $v\in V(G)$. 
    \vspace{2mm}
    
    Let $L^{1}=\{L^{1}(v)\}_{v\in V_{G[X_{1}]}}$. Since $G[X_{1}]$ is properly $L^{1}$-distinguishable (say by coloring $f^1$), we define $f\restriction V_{G[X_{1}]}:=f^{1}$.
    \vspace{2mm}
    
    \item For any $1<i\leq t$, let $L^{i}(v):=$ init$_{k_{i}}(L(v)\backslash \bigcup_{k= 1}^{i-1} L^{k}(v)) $ for each vertex $v\in V(G)$. 
    \vspace{2mm}
    
    Let $L^{i}=\{L^{i}(v)\}_{v\in V_{G[X_{i}]}}$. Since $G[X_{i}]$ is properly $L^{i}$-distinguishable (say by coloring $f^i$), we define $f\restriction V_{G[X_{i}]}:=f^{i}$.
\end{enumerate}  

If $\phi\in Aut(G)$ preserves $f$, then the range of $G[X_{i}]$ with respect to $\phi$ is $G[X_{i}]$, for all $1\leq i\leq t$. Thus, $\phi\restriction G[X_{i}]$ is an automorphism of $G[X_{i}]$ preserving $f$. Since the above defined colorings of $G[X_{i}]$'s are distinguishing, we have that $\phi\restriction G[X_{i}]$ is a trivial automorphism for all $1\leq i\leq t$. Thus, $f$ is a distinguishing coloring of $G$. 
To show $f$ is a proper coloring of $G$, we pick any $x,y\in V(G)$ such that $\{x,y\}\in E_{G}$. If $x,y\in V_{G[X_{i}]}$ for some $1\leq i\leq t$, then $f(x)\neq f(y)$, as $f\restriction (V_{G[X_{i}]})$ is a proper coloring. If $x\in V_{G[X_{i}]}$ and $y\in V_{G[X_{j}]}$ such that $i\neq j$, then $f(x)\neq f(y)$, as $f$ is defined in a way such that the sets of colors used to color $G[X_{i}]$ and $G[X_{j}]$ are different.
This concludes the proof that $\chi_{D_{L}}(G)\leq \sum_{i=1}^{t} \chi_{D_{L}}(G[X_{i}])$.
By applying Claim 
\ref{claim 3.3}, we obtain
\begin{equation*}
\begin{aligned}
 \chi_{D_{L}}(G) 
 &\leq \sum_{i=1}^{t} \chi_{D_{L}}(G[X_{i}])
       \leq \sum_{i=1}^{t}(2h_{i}-1)
        = 2\sum_{i=1}^{t} h_{i}-t
        =2(h-t+2)-t 
        = 2\Delta(G)- (3t-4).
\end{aligned}
\end{equation*}
\end{proof} 
\section{Applications}

Throughout the section, we fix integers $r>6$ and $n$ with $3r+3\le n\le r^2$. We define $a=n-2r+1$ and $b=2r-1$. 

\begin{thm}\label{Theorem 4.1}
{\em Let $X$ be any graph on $a=n-2r+1$ vertices satisfying the following properties:
\begin{enumerate}[(i)]
    \item $\alpha(X)\le r-1$, 
    \item $\kappa(X)\ge a-b$, and
    \item $\Delta(X)\ge r+2$.
\end{enumerate}

Let $H$ be any graph on $b$ vertices with $\alpha(H)\le r-1$. Then the join $G=X\vee H$ of $X$ and $H$ satisfies the following:
\begin{enumerate}
    \item $G$ is $(n-2r+1)$-connected,
    \item $G$ contains no induced subgraph isomorphic to $K_{r,r+1}$,
    \item $\Delta(G)\ge 3r+1$. Thus, $\displaystyle \frac{\Delta(G) + 2}{r + 1} \ge 3$.
\end{enumerate}
Consequently, by Theorem \ref{Theorem 1.7}, we obtain $\chi_{D_{L}}(G)\leq 2\Delta(G)-5$ which is a better bound for $\chi_{D_{L}}(G)$ than $\chi_{D_{L}}(G)\leq 2\Delta(G)-1$.
}
\end{thm}

\begin{proof}
The bipartitions of an induced subgraph $K_{r,r+1}$ are independent sets, and in the join $G = X \vee H$ any such bipartition must lie entirely in $X$ or entirely in $H$. Hence, the bounds $\alpha(X) \le r-1$ and $\alpha(H) \le r-1$ prevent $G$ from containing $K_{r,r+1}$ as an induced subgraph.
Since $G=X\lor H$, we have 
\begin{center}
$\kappa(G)=\min\{\kappa(X)+b,\kappa(H)+a\}$.
\end{center}
Since $\kappa(X)\ge a-b$ by assumption, we have $\kappa(X)+b\geq a$. Thus, $\kappa(G)\geq a$.
Consequently, $G$ is $a$-connected.
Finally, if $x\in V(X)$ such that $\deg_X(x)\ge r+2$ (possible by our assumption $\Delta(X)\geq r+2$), then
$\deg_G(x)=\deg_X(x)+b\geq (r+2)+(2r-1)= 3r+1$. Thus, $\Delta(G)\ge 3r+1$.
\end{proof}

\begin{example}\label{Example 4.2}
The following shows that $X=K_{a}$ satisfies $(i)-(iii)$ of Theorem \ref{Theorem 4.1}.

\begin{enumerate}[$(i)$]
    \item Since $\alpha(X) = 1$ and $r>6$, we have $\alpha(X)< r-1$.
    
    \item Since $b=2r-1> 11$, we have $\kappa(X) = a-1> a-b$. 
    
    \item Clearly, $\Delta(X) = a-1=n-2r\geq (3r+3)-2r=r+3> r+2$.
\end{enumerate}
\end{example}
\subsection{Paley graph}

\begin{prop}\label{Proposition 4.3}
{\em Let $a$ be a power of a prime with
$a\equiv1\pmod4$ satisfying 
$2r+5 \leq a \leq 4r-3$.
Then the Paley graph $X=P(a)$ satisfies $(i)-(iii)$ of Theorem \ref{Theorem 4.1}.}
\end{prop}

\begin{proof}
The graph $P(a)$ is $(\frac{a-1}{2})$-regular on $a$ vertices with the least eigenvalue $\frac{-1-\sqrt a}{2}$. Thus,
\[
\alpha(P(a))
\le a\cdot\frac{-(\frac{-1-\sqrt a}{2})}
{\frac{a-1}{2} - (\frac{-1-\sqrt a}{2})}
= a\cdot\frac{1+\sqrt a}{a+\sqrt a}
= \sqrt a
\]

by Lemma \ref{Lemma 2.2}.  
Since $a\le4r-3$ and $r>6$, we have
$4r-3<(r-1)^2$, and hence
$a\le(r-1)^2$.
Thus, we obtain $\alpha(X)\le r-1$.
Now, $X$ is connected, edge-transitive, and $(\frac{a-1}{2})$-regular.  By Lemma \ref{Lemma 2.7}, $\kappa(X)=\delta(X)=\frac{a-1}{2}$. Since $4r-3\geq a$ and $b=2r-1$, we obtain $\kappa(X)\geq a-b$.
Since $X$ is $(\frac{a-1}{2})$-regular and $a\geq 2r+5$, we have $\Delta(X)=\frac{a-1}{2}\geq r+2$. 
\end{proof}
\subsection{Cayley graph of the dihedral group}

Suppose $a=2m\geq r+8$ is an even integer and $1\leq k\leq 5$ is any integer. Let 
\begin{itemize}
    \item $D_{2m}=\langle \rho,\sigma : \rho^m=\sigma^2=e,\ \sigma\rho\sigma=\rho^{-1}\rangle$ 
be the dihedral group of order $2m$,
    \item $R=\{\rho^j:0\le j\le m-1\}$ be the set of all rotations, 
    \item $R\sigma=\{\rho^j\sigma:0\le j\le m-1\}$ be the set of all reflections, and
    \item $F\subseteq R\sigma$ satisfy $|F|=m-k$.
\end{itemize}

In this section, we show that if $S = R\setminus\{e\} \cup F$ is a generating set,
then $X=\mathrm{Cay}(D_{2m},S)$ satisfies $(i)-(iii)$ of Theorem \ref{Theorem 4.1}. The following example illustrates the construction of such a Cayley graph.

\begin{figure}[ht]
\centering
\begin{tikzpicture}[
    scale=0.9,
    every node/.style={circle,draw,fill=white,inner sep=1.5pt,font=\scriptsize},
    >=stealth
]

\draw[rounded corners=6pt, thick] (-2.2,-4) rectangle (2.4,4);
\node[draw=none] at (0.6,4.2) {$R\cong K_8$};

\draw[rounded corners=6pt, thick] (6.0,-4) rectangle (10.6,4);
\node[draw=none] at (7.8,4.2) {$R\sigma\cong K_8$};

\foreach \i/\y in {
0/3.5,
1/2.5,
2/1.5,
3/0.5,
4/-0.5,
5/-1.5,
6/-2.5,
7/-3.5}
{
\node (R\i) at (0.6,\y) {$\rho^{\i}$};
}

\node[draw=none,left] at (R0.west) {$e=$};

\foreach \i/\y in {
0/3.5,
1/2.5,
2/1.5,
3/0.5,
4/-0.5,
5/-1.5,
6/-2.5,
7/-3.5}
{
\ifnum\i=0
\node (S0) at (7.8,\y) {$\sigma$};
\else
\node (S\i) at (7.8,\y) {$\rho^{\i}\sigma$};
\fi
}


\foreach \i in {0,...,7}{
    \foreach \j in {0,...,5}{
        \pgfmathtruncatemacro{\k}{mod(\i+\j,8)}
        \draw[gray] (R\i)--(S\k);
    }
}

\foreach \i in {0,...,6}{
    \pgfmathtruncatemacro{\j}{\i+1}
    \draw[black] (R\i) -- (R\j);
    \draw[black] (S\i) -- (S\j);
}

\foreach \i in {0,...,5}{
    \pgfmathtruncatemacro{\j}{\i+2}
    \draw[black] (R\i) to[bend right=35] (R\j);
    \draw[black] (S\i) to[bend left=35] (S\j);
}

\foreach \i in {0,...,4}{
    \pgfmathtruncatemacro{\j}{\i+3}
    \draw[black] (R\i) to[bend right=45] (R\j);
    \draw[black] (S\i) to[bend left=45] (S\j);
}

\foreach \i in {0,...,3}{
    \pgfmathtruncatemacro{\j}{\i+4}
    \draw[black] (R\i) to[bend right=55] (R\j);
    \draw[black] (S\i) to[bend left=55] (S\j);
}

\foreach \i in {0,...,2}{
    \pgfmathtruncatemacro{\j}{\i+5}
    \draw[black] (R\i) to[bend right=65] (R\j);
    \draw[black] (S\i) to[bend left=65] (S\j);
}

\foreach \i in {0,...,1}{
    \pgfmathtruncatemacro{\j}{\i+6}
    \draw[black] (R\i) to[bend right=75] (R\j);
    \draw[black] (S\i) to[bend left=75] (S\j);
}

\foreach \i in {0}{
    \pgfmathtruncatemacro{\j}{\i+7}
    \draw[black] (R\i) to[bend right=85] (R\j);
    \draw[black] (S\i) to[bend left=85] (S\j);
}
\end{tikzpicture}

\caption{The graph $\mathrm{Cay}(D_{16},S)$ from Example~\ref{Example 4.3}, where
$S=R\setminus\{e\}\cup
\{\sigma,\rho\sigma,\rho^2\sigma,\rho^3\sigma,\rho^4\sigma,\rho^5\sigma\}$.
Each of $R$ and $R\sigma$ induces a clique $K_8$.
}
\label{Figure 3}
\end{figure}

\begin{example}\label{Example 4.3}
Let $r=7$ and $n=29$. Then
\[
a=n-2r+1=16\quad\text{ and } \quad b=2r-1=13.
\]
Thus $a=2m$ with $m=8$. Choose $k=2$ and consider the dihedral group
\[
D_{16}=\langle \rho,\sigma:\rho^8=\sigma^2=e,\ \sigma\rho\sigma=\rho^{-1}\rangle.
\]
Further, let
$
R=\{e,\rho,\rho^2,\ldots,\rho^7\},
$
$
R\sigma=\{\sigma,\rho\sigma,\rho^2\sigma,\ldots,\rho^7\sigma\},
$
and
$
F=\{\sigma,\rho\sigma,\rho^2\sigma,\rho^3\sigma,\rho^4\sigma,\rho^5\sigma\}.
$
Then $|F|=6=m-k$, and
\[
S=R\setminus\{e\}\cup F
\]
is a generating set of $D_{16}$ since $\rho,\sigma\in S$. Figure~\ref{Figure 3} illustrates the Cayley graph
$
X=\mathrm{Cay}(D_{16},S),
$
which consists of two cliques induced by $R$ and $R\sigma$, together with edges between them determined by the generating set $F$.
\end{example}

\begin{prop}\label{Proposition 4.4}
{\em Suppose $a=2m\ge r+8$ is an even integer and $1\le k\le5$ is any integer. Let $F$ be a set of reflections of cardinality $m-k$. If $S=(R\setminus\{e\})\cup F$
is a generating set of $D_{2m}$, then the Cayley graph
$X=\mathrm{Cay}(D_{2m},S)$
satisfies (i)--(iii) of Theorem~\ref{Theorem 4.1}.}
\end{prop}

\begin{proof}
Let $d=|S|=(m-1)+(m-k)=2m-(k+1)$. Then $X$ is a $d$-regular graph.

\begin{enumerate}[(i)]
    \item We write $D_{2m}=R\cup R\sigma$ as the disjoint union of cosets $R$ and $R\sigma$.
    Now, each of $R$ and $R\sigma$ induces a clique in $X$.
Take two distinct rotations \(\rho^i,\rho^j\in R\) with \(i< j\). Then
$
(\rho^i)^{-1}\rho^j=\rho^{j-i}\in R\setminus\{e\}\subseteq S,
$
so \(\rho^i\) and \(\rho^j\) are adjacent in $X$.
Similarly, take two distinct reflections \(\rho^i\sigma,\rho^j\sigma\in R\sigma\) with \(i> j\). 
We note that $\sigma^{2}=e$, so $\sigma^{-1}=\sigma$ and $\sigma\rho^{k}\sigma=\rho^{-k}$ for all integers $k\geq 1$.
Then
\[
(\rho^i\sigma)^{-1}(\rho^j\sigma)
= (\sigma^{-1}\rho^{-i})(\rho^j\sigma) 
= \sigma\rho^{-i}\rho^j\sigma
= \sigma\rho^{\,j-i}\sigma
= \rho^{-(j-i)}=\rho^{\,i-j}\in S.
\]

So, \(\rho^i\sigma\) and \(\rho^j\sigma\) are adjacent in $X$.
Thus, any independent set in $X$ can contain at most one rotation and at most one reflection. Hence $\alpha(X)\leq 2< r-1$ since $r> 6$.

\item We show that $\kappa(X)\geq 2m-(2k+1)$. By Menger's theorem, it suffices to show that for any two distinct nonadjacent vertices
$u,v\in V(X)$, there are at least $2m-(2k+1)$ internally vertex-disjoint $u$--$v$ paths.
By the arguments in (i), $R$ and $R\sigma$ induce cliques isomorphic to $K_m$ in $X$. Hence, the only possible nonadjacent pair consists of one rotation and one reflection. Moreover, each rotation has exactly $m-k$ neighbors in $R\sigma$, and each reflection has exactly $m-k$ neighbors in $R$.
Let $u=\rho^{i}\in R$ and $v=\rho^{j}\sigma\in R\sigma$ be two nonadjacent vertices. Since $R$ induces a clique, every neighbor $x\in R$ of $v$ gives rise to the path
\[
P_x=\{ux,xv\}.
\]
As $v$ has exactly $m-k$ neighbors in $R$, this yields $m-k$ internally vertex-disjoint paths.
Similarly, since $R\sigma$ induces a clique, every neighbor $s\in R\sigma$ of $u$ is adjacent to $v$. Hence, for each such reflection $s$, we obtain the path
\[
P_s=\{us,sv\}.
\]
Since $u$ has exactly $m-k$ neighbors in $R\sigma$, these contribute another $m-k$ internally vertex-disjoint paths.
Altogether, we obtain
\[
2(m-k)=2m-2k>2m-(2k+1)
\]
internally vertex-disjoint $u$--$v$ paths. Thus, $\kappa(X)\ge 2m-(2k+1)=a-(2k+1)> a-b$ since $b=2r-1>2k+1$.

\medskip

    \item Since $X$ is $d$-regular and $a\geq r+8$, we have $\Delta(X)=2m-(k+1)=a-(k+1)\geq r+2$.
\end{enumerate}
\end{proof}

\end{document}